\documentclass[a4paper,11pt]{article}
\usepackage{fullpage,amsmath,amsthm,amssymb,enumerate,tikz}

\newtheorem{lemma}{Lemma}[section]

\newtheorem{theorem}[lemma]{Theorem}
\newtheorem{prop}[lemma]{Proposition}
\theoremstyle{definition}
\newtheorem{question}[lemma]{Question}
\theoremstyle{remark}

\newcommand*{\abs}[1]{\lvert #1\rvert}
\newcommand*{\floor}[1]{\lfloor #1\rfloor}
\newcommand{\ie}{i.e.\ }

\title{Determining triangulations and quadrangulations by boundary distances}
\author{John Haslegrave}

\begin{document}
\maketitle
\begin{abstract}We show that if all internal vertices of a disc triangulation have degree at least $6$, then the full structure can be determined from the pairwise graph distances between boundary vertices. A similar result holds for disc quadrangulations with all internal vertices having degree at least $4$. This confirms a conjecture of Itai Benjamini. Both degree bounds are best possible, and correspond to local non-positive curvature. However, we show that a natural conjecture for a ``mixed'' version of the two results is not true.
\end{abstract}

\section{Introduction}\label{sec:intro}
Let $T$ be a planar near-triangulation (that is, a planar graph such that every face other than the outer face is a triangle) with a simple closed boundary of length $n$, and suppose that all distances in $T$ between boundary vertices are known. Under what conditions does this allow the precise structure of $T$ to be determined? In general this is not the case if internal vertices of degree $5$ (or below) are permitted. For any near-triangulation we can glue an icosahedron into any face, adding nine vertices of degree $5$, without changing any distances between existing vertices. This question was first posed by Itai Benjamini \cite{Itai}, who conjectured that provided all internal vertices have degree at least $6$, the near-triangulation can be reconstructed.

Indeed, degree $5$ is the natural boundary to this sort of counterexample. If all internal vertices have degree at least $6$, then no triangle can have internal vertices. To see this, suppose $abc$ is a triangle surrounding $m\geq 1$ internal vertices. Then the total number of edges between and inside $abc$ is at least $3m+9/2$, since each of $a,b,c$ has at least one neighbour inside the triangle, but a planar graph on $m+3$ vertices can have at most $3m+3$ edges. 

Similarly, we might instead consider a planar near-quadrangulation. If internal vertices of degree $3$ are permitted then we may glue a cube onto a face, but again this type of example cannot occur if the minimum degree of internal vertices is at least $4$. Both of these conditions (minimum degree $6$ for triangulations and minimum degree $4$ for quadrangulations) correspond to having non-positive curvature at each internal vertex.

In more generality, if we consider a planar map with simple closed boundary which may have faces of different sizes, then the natural interpretation of what it means to have non-positive curvature at each internal vertex is to look locally at the faces surrounding a vertex, considering each as a regular polygon with the appropriate number of sides, and require the angles around the vertex to sum to at least $2\pi$. For example, if all faces are triangles or quadrangles, and an internal vertex $v$ meets $t(v)$ triangles and $q(v)$ quadrangles, then this condition becomes $2t(v)+3q(v)\geq 12$. As we shall see, this is indeed the right definition, in the sense that it precludes the type of simple counterexamples discussed above. However, perhaps surprisingly in light of our main results, this condition is not sufficient to allow reconstruction of the graph from boundary distances, even for graphs with only these two types of faces; a class of counterexamples is given in Section \ref{sec:counter}.

However, we show that in the cases of pure near-triangulations and pure near-quadrangulations the natural condition is sufficient to allow reconstruction.
\begin{theorem}\label{main:tri}Let $T$ be a near-triangulation with a simple closed boundary such that all internal vertices have degree at least $6$. Then the distances between boundary vertices determine $T$.
\end{theorem}
\begin{theorem}\label{main:quad}Let $Q$ be a near-quadrangulation with a simple closed boundary such that all internal vertices have degree at least $4$. Then the distances between boundary vertices determine $Q$.
\end{theorem}

Note that the information we receive in each case is the set of boundary vertices and all pairwise distances between them. In particular, we are \textit{not} given the cyclic order of boundary vertices. However, this information can easily be recovered as an initial step.\footnote{I am grateful to Alex Scott for this observation.}
\begin{lemma}\label{lem:hamilton}Let $G$ be any planar graph with simple closed boundary. Then the set of pairs of boundary vertices at distance $1$ determines the cyclic ordering of boundary vertices.
\end{lemma}
\begin{proof}Let $G'$ be the subgraph induced on the boundary vertices, and consider an edge $uv\in E(G')$. If $uv$ is a boundary edge, \ie $u$ and $v$ are consecutive boundary vertices, then $G'\setminus\{u,v\}$ is connected, since the remaining boundary edges form a spanning path. However, if $uv$ is a chord then no other chord can cross it, and so $G'\setminus\{u,v\}$ has two components consisting of the two sections of the boundary between $u$ and $v$.
\end{proof}

Theorems \ref{main:tri} and \ref{main:quad} may be viewed as discrete versions of the continuous inverse problem of reconstructing (up to isometry) the metric on a compact Riemannian metric with specified boundary from the distance function on pairs of boundary points. This problem originates from rigidity questions in Riemannian geometry \cite{Mic81,Cro91,Gro91}, and a manifold is said to be \textit{boundary rigid} if any other Riemannian manifold with the same boundary and boundary distance function is isometric to it. Michel \cite{Mic81} conjectured that any simple Riemannian manifold is boundary rigid; the condition of being simple requires that the manifold is simply-connected with strictly convex boundary and that geodesics have no conjugate points. In the two-dimensional case this was confirmed for simple surfaces of negative curvature by Croke \cite{Cro90} and for all simple surfaces by Pestov and Uhlmann \cite{PU05}. Note, however, that these results in the continuous case require strong assumptions on the boundary, whereas our results hold for any boundary which encloses a simply-connected region.

In addition to the reconstruction question, we might also ask about recognition: is there an efficient way to determine whether a given list of distances arise from such a graph? While the proofs of our main results do indeed give a way to do this, the recognition question also makes sense in contexts where reconstruction is not possible. In Section \ref{sec:algo} we discuss such generalisations.

Throughout, we write $d_G(\cdot,\cdot)$ for the graph distance within $G$ (or just $d(\cdot,\cdot)$ if there is only one possible graph); to avoid confusion, we write $\deg_G(\cdot)$ (or $\deg(\cdot)$) for the vertex degree.

\section{Disc triangulations}
In this section we prove Theorem \ref{main:tri}. Our approach is induction on $n$, the number of boundary vertices (the observation in Section \ref{sec:intro} that every triangle is a face gives the result for $n=3$). We show that certain types of configuration allow reduction to smaller cases. For a configuration to allow such a reduction, we require two properties: first, that the presence of this configuration within $T$ can be recognised from the boundary distances; secondly, that after removing the configuration what remains consists of one or more smaller disc triangulations, and all boundary distances for these may be deduced from the original boundary distances. By Lemma \ref{lem:hamilton}, we may assume that the cyclic ordering of boundary vertices is known.

A simple example of such a configuration is a chord, that is, if $T$ contains an internal edge between two boundary vertices. A chord may be identified because its ends are a non-consecutive pair of boundary vertices at distance $1$, giving the first property. Assuming such a chord exists, let the boundary vertices, in order, be $v_1,\ldots,v_n$ and suppose without loss of generality that $d(v_1,v_k)=1$ for $k\neq 2,n$. Then we may split $T$ into two near-triangulations $T_1$ and $T_2$ bounded respectively by the cycles $v_1\cdots v_k v_1$ and $v_k\cdots v_nv_1v_k$. No shortest path (in $T$) between two vertices of $T_1$ can use vertices outside $T_1$, since if a path leaves $T_1$ and subsequently returns to it then the section outside $T_1$ can be replaced by the single edge $v_1v_k$ (in some direction). Thus for each pair $(v_i,v_j)$ of boundary vertices of $T_1$ we have $d_{T_1}(v_i,v_j)=d_T(v_i,v_j)$. The same applies to $T_2$, giving the second property.

Of course, in order to make this approach into an inductive proof, we will need our set of reducible configurations to be sufficiently large that at least one of them is guaranteed to occur (except in the base case). While there are some other ``small'' configurations (in the sense of depending only on one or two internal vertices) that allow reductions, in general we will need to consider configurations of unbounded size. The general form of these will be a strip of faces running along the boundary, which we can remove to leave a single near-triangulation consisting of all remaining faces.

A key fact which we will need in order to deduce the boundary distances of the smaller near-triangulation is that in any near-triangulation where internal vertices all have degree at least $6$, the three vertices of any face cannot be equidistant from any fourth point; we prove this in Section \ref{sec:meridian}. We then define the main configurations we use in Section \ref{sec:nice}, and show that they may be identified and permit reductions. After giving some other simpler reduction configurations in Section \ref{sec:misc}, we complete the proof in Section \ref{sec:main}.

In what follows, a \textit{disc triangulation} is a plane graph with a simple closed boundary and all internal faces being triangles. It is \textit{chordless} if no internal edge joins two boundary vertices. We define the \textit{boundary faces} of $T$ to be those which share an edge with the boundary cycle.

\subsection{Meridians and no equidistant triangle}\label{sec:meridian}

\begin{lemma}\label{geodesic}Let $T$ be a disc triangulation with boundary $v_0v_1\cdots v_{n-1}v_0$ such that every internal vertex has degree at least $6$. Suppose that $v_1,\ldots,v_{k-1}$ each have degree at least $4$. Then $d_T(v_0,v_k)=k$, and furthermore $v_0v_1\cdots v_k$ is the unique shortest path between $v_0$ and $v_k$.
\end{lemma}
\begin{proof}
First we argue that for every pair $0\leq i<j\leq k$ with $j>i+1$, the vertices $v_i,v_j$ are not adjacent and do not have a common neighbour other than $v_{i+1}$. Indeed, suppose not and consider a counterexample minimising $j-i$. Consider the boundary faces with edges $v_iv_{i+1},\ldots,v_{j-1}v_j$. The third vertices of all these faces are distinct, using minimality and the fact that each of $v_1,\ldots,v_{k-1}$ has degree at least $4$. Thus if $v_i$ and $v_j$ are adjacent then the cycle $v_i\cdots v_j$ is chordless and encloses a triangulation where all internal vertices have degree at least $6$ and at least $j-i-1$ internal vertices are adjacent to the boundary. Similarly, if $v_i$ and $v_j$ have a common neighbour $x$ (other than $v_{i+1}$), then the cycle $v_i\cdots v_jx$ has the same properties. In either case, it follows that the inner subtriangulation remaining after removing the cycle of length at most $j-i+1$ has boundary length at least $j-i-2$. However, this gives a contradiction together with \cite[Lemma 2.2]{ABH}, which states that if $T'$ is a triangulation of a cycle of order $\ell$ with all internal vertices having degree at least $6$ then the inner subtriangulation remaining after removing the cycle has boundary length at most $\ell-6$.

Now we prove the desired result. Suppose it is not true and take a counterexample with as few total vertices in the near-triangulation as possible, and among such counterexamples, with $k$ as small as possible. Let $P$ be either a shorter path or a different path of length $k$ between $v_0$ and $v_k$. If $v_j$ is on $P$ for some $1\leq j<k$ then either the section of $P$ from $v_0$ to $v_j$ or the section from $v_j$ to $v_k$ contradicts minimality of the original counterexample. Thus $P$ avoids $v_1,\ldots,v_{k-1}$.

Consider the set of faces incident with at least one of $v_1,\ldots,v_{k-1}$. The number of such faces is $k+\sum_{i=1}^{k-1}(\deg(v_i)-3)\geq 2k-1$. There is another path $P'$ from $v_0$ to $v_k$ around the other side of this strip of faces, which has length at least $k+1$ and is disjoint from $\{v_1,\ldots,v_{k-1}\}$; in particular $P'$ is longer than $P$, so is not a shortest path in $T\setminus\{v_1,\ldots,v_{k-1}\}$. Suppose that $P'$ consists entirely of internal vertices of $T$. Then removing $v_1,\ldots,v_{k-1}$ leaves a disc triangulation $T'$ for which $P'$ forms part of the boundary, and since each vertex on $P'$, other than $v_0$ and $v_k$, was adjacent to at most two of the removed vertices, we conclude that it has degree at least $4$ in the remaining near-triangulation. Since $P$ lies entirely in $T'$, it follows that $P'$ is not a shortest path between $v_0$ and $v_k$ in $T'$, contradicting minimality of the original example. If some vertices of $P'$ are boundary vertices of $T$, we may remove any bridges from $T'$ and split any cutvertices to leave two or more disc triangulations, then apply the same argument to each section of $P'$ that remains, noting that every vertex of $P'$ which lies strictly inside such a section has degree at least $4$.
\end{proof}
\begin{lemma}\label{geodesic2}Let $T$ be a triangulation of the cycle $v_0v_1\cdots v_{n-1}v_0$ with every internal vertex having degree at least $6$. Suppose that $v_1,\ldots,v_{k-1}$ each have degree at least $3$, that if $\deg(v_i)=3$ then there is some $h$ with $0<h<i$ and $\deg(v_h)\geq 5$, and that if there are two values $0<i<j<k$ with $\deg(v_i)=\deg(v_j)=3$ then there is some $h$ with $i<h<j$ and $\deg(v_h)\geq 5$. Then $d_T(v_0,v_k)=k$, and furthermore every path of length exactly $k$ uses $v_1$.
\end{lemma}
\begin{proof}Again, suppose not and take a counterexample with as few total vertices as possible. By Lemma \ref{geodesic}, at least one of these vertices, say $v_i$, has degree $3$. Let $w_i$ be the neighbour of $v_i$ other than $v_{i-1},v_{i+1}$, and remove $v_i$ to give a near-triangulation $T'$. For any shortest $v_0v_k$-path in $T$ which uses $v_i$, replacing $v_i$ by $w_i$ gives another shortest path, and so it suffices to show the conclusion holds for $T'$.

Note that $v_0,\ldots,v_{i-1},w_i,v_{i+1},\ldots v_k$ are consecutive boundary vertices of $T'$, $\deg_{T'}(v_{i-1})=\deg_T(v_{i-1})-1$, $\deg_{T'}(v_{i+1})=\deg_T(v_{i+1})-1$ and $\deg_{T'}(w_i)=\deg_T(w_i)-1$, with all other degrees unchanged. If $w_i$ was an internal vertex of $T$ then $T'$ is still a disc triangulation and $\deg_{T'}(w_i)\geq 5$. It follows that $T'$ satisfies the same boundary degree conditions, and so the result holds by minimality of $T$. If not, then $T'$ consists of two disc triangulations $T'_1$ and $T'_2$ with a common vertex $w_i$. The required boundary conditions are met in $T'_1$ for vertices $v_1,\ldots,v_{i-1}$ and in $T'_2$ for $v_{i+1}\ldots v_{k-1}$, and so $d_{T'_1}(v_0,w_i)=i$ with every shortest path using $v_1$, and $d_{T'_2}(w_i,v_k)=i-k$, from which the result follows.
\end{proof}
We work towards the crucial Lemma \ref{no-equal-triangle}. We first need a definition.

Given a plane near-triangulation $T$, and an internal vertex $x$, let $C$ be the cycle of neighbours of $x$. We regard $C$ as a cyclic list alternating between vertices and edges, of even length $2\deg(x)$. If two elements of this list have positions differing by $\deg(x)$, we say they are \textit{opposite} with respect to $x$. If $y,z$ are two neighbours of $x$ then their positions in the list differ by an even number, so there are two list elements exactly halfway between them which form a pair of opposite elements; we say either such element is a \textit{midpoint} of $y$ and $z$ with respect to $x$. We may similarly define the midpoint of $y$ and $z$ relative to $x$ when $x$ is not internal; in this case there is only one midpoint. 

We define a \textit{meridian} in $T$ as a sequence $e_0,\ldots,e_r$ with the following properties:
\begin{itemize}
\item For each $i$, $e_i$ is either a vertex or an edge.
\item $e_i$ is a boundary vertex or boundary edge if and only if $i\in\{0,r\}$.
\item For $0<i<r$, if $e_i$ is a vertex then $e_{i-1}$ and $e_{i+1}$ are opposite with respect to $e_i$.
\item For $0<i<r$, if $e_i$ is an edge then $e_{i-1}$ and $e_{i+1}$ are the two vertices completing faces with $e_i$ in some order.
\end{itemize}
Note that we may associate a meridian with a simple curve, passing through each vertex and crossing each edge which is an element of the meridian, meeting the boundary of $T$ (only) at each endpoint. Consequently a meridian divides $T$ into two parts.

We say that a meridian \textit{bisects} two incident edges $xy$ and $xz$ if $x$ is an element of the meridian and the adjacent elements are midpoints of $y$ and $z$ relative to $x$. Note that this includes the case where $x$ is a boundary vertex and so there is only one adjacent element (and only one midpoint).
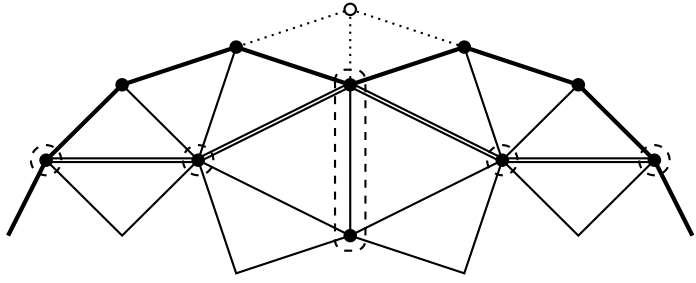
\begin{figure}
\centering
\begin{tikzpicture}[thick]
\draw[dotted] (2.5,1.5) -- (4,2) -- (5.5,1.5);
\draw[dotted] (4,2) -- (4,1);
\draw[ultra thick] (-.5,-1) -- (0,0) -- (1,1) -- (2.5,1.5) -- (4,1) -- (5.5,1.5) -- (7,1) -- (8,0) -- (8.5,-1);
\draw[double] (0,0) -- (2,0) -- (4,1) -- (6,0) -- (8,0);
\draw (1,1) -- (2,0) -- (2.5,1.5);
\draw (7,1) -- (6,0) -- (5.5,1.5);
\draw (0,0) -- (1,-1) -- (2,0) -- (2.5,-1.5) -- (4,-1) -- (5.5,-1.5) -- (6,0) -- (7,-1) -- (8,0);
\draw (2,0) -- (4,-1) -- (6,0);
\draw (4,-1) -- (4,1);
\filldraw[fill=white] (4,2) circle (0.075);
\filldraw (0,0) circle (0.075);
\filldraw (2,0) circle (0.075);
\filldraw (1,1) circle (0.075);
\filldraw (2.5,1.5) circle (0.075);
\filldraw (4,1) circle (0.075);
\filldraw (4,-1) circle (0.075);
\filldraw (5.5,1.5) circle (0.075);
\filldraw (7,1) circle (0.075);
\filldraw (6,0) circle (0.075);
\filldraw (8,0) circle (0.075);
\draw[dashed] (0,0) circle (0.2);
\draw[dashed] (2,0) circle (0.2);
\draw[dashed] (6,0) circle (0.2);
\draw[dashed] (8,0) circle (0.2);
\draw[dashed, rounded corners] (3.8,1.2) rectangle (4.2,-1.2);
\end{tikzpicture}
\caption{Part of a near-triangulation, with a meridian (circled groups) and a path following the meridian (doubled lines). The dotted section at the top shows how the near-triangulation may be expanded to avoid the path meeting the boundary (bold) internally.}\label{fig:line}
\end{figure}
Since each $e_i$ is either a vertex or an edge, we abuse notation by writing $v\in e_i$ to mean that either $e_i$ is the vertex $v$ or is an edge containing $v$. We say that a path $v_0v_1\cdots v_s$ in $T$ \textit{follows} a meridian $e_0,\ldots,e_r$ if $v_0=e_i$ for some $i$, and either for every $j$ we have $v_j\in e_{i+j}$ or for every $j$ we have $v_j\in e_{i-j}$. Note that we require the path to start in a vertex of the meridian. Figure \ref{fig:line} shows a meridian and following path. Our next lemma shows that paths which follow meridians are shortest paths for the class of near-triangulations we consider.
\begin{lemma}\label{following}Let $T$ be a disc triangulation for which every internal vertex has degree at least $6$, let $\ell$ be a meridian in $T$ and let $v_0v_1\cdots v_s$ be a path which follows $\ell$. Set $v'_1$ to be the other vertex (if it exists) for which $v_0v_1'v_2$ follows $\ell$. Then $d_T(v_0,v_s)=s$, and furthermore if $v_0w_1\cdots w_{s-1}v_s$ is any path of length $s$ then $w_1\in \{v_1,v'_1\}$.
\end{lemma}
\begin{proof}Suppose not, and take a counterexample $P=v_0v_1\cdots v_s$ following a meridian $\ell$, with $s$ as small as possible. Note that we may assume $v_1,\ldots,v_{r-1}$ are internal vertices of $T$, since we may add vertices to $T$, if necessary, to ensure this without violating the fact that $P$ is a counterexample, and in such a way that all internal vertices have degree at least $6$ (see Figure \ref{fig:line}).

Let $P'$ be an alternative path which demonstrates that $P$ is a counterexample, and let $\gamma$ be a curve associated with $\ell$. Note that minimality ensures that, for $0<i<r$, if $v_i$ is an element of $\ell$ then $P'$ does not visit that vertex, and if $v_iv_i'$ is an element of $\ell$ then $P'$ does not use that edge. In particular, this means that $P'$ does not cross $\gamma$, so we may choose a subtriangulation $T'$ induced by vertices which are either on $\gamma$ or on a particular side of it, such that $P'$ lies entirely in $T'$. We may assume $P$ also lies entirely within $T'$ (if not, replace each vertex which does not with its neighbour which does). We may also assume $T'$ is a disc triangulation, since our previous assumption that $v_1,\ldots,v_{r-1}$ were internal in $T$ ensures that $P$ and $P'$ lie entirely within one block of $T'$.

Fix $i$ with $0<i<r$. Suppose first that $v_i$ is an element of $\ell$. If $v_{i-1}$ and $v_{i+1}$ are both elements of $\ell$ then they are opposite in $T$ with respect to $v_i$, and since $\deg_T(v_i)\geq 6$ we have $\deg_{T'}(v_i)\geq 4$. If exactly one is an element of $\ell$ then $\deg_T(v_i)$ must be odd, so at least $7$, and again $\deg_{T'}(v_i)\geq 4$. Finally, if neither is an element of $\ell$ then we obtain $\deg_{T'}(v_i)\geq 3$.

Alternatively, if $v_i$ is not an element of $\ell$ then it has exactly one neighbour in $T\setminus T'$, and so $\deg_{T'}(v_i)\geq 5$. 

Note that each $v_i$ satisfies $\deg_{T'}(v_i)\geq 3$, and if $\deg_{T'}(v_i)=3$ it follows that $i>1$ and $\deg_{T'}(v_{i-1})\geq 5$. Thus the conditions of Lemma \ref{geodesic2} are met, giving the result.\end{proof}

We can now prove the main result of this section.
\begin{lemma}\label{no-equal-triangle}Let $T$ be a near-triangulation with a simple closed boundary with all internal vertices having degree at least $6$. Let $a,b,c$ be the vertices of a triangle in $T$, and let $x$ be any other vertex of $T$. Then the distances $d_T(x,a),d_T(x,b),d_T(x,c)$ are not all equal.
\end{lemma}
\begin{proof}
Suppose the result is not true, and take a counterexample where the common distance is as small as possible; note that it cannot be $1$ because each triangle bounds a face. Fix three shortest paths $P_a,P_b,P_c$ to $a,b,c$ respectively, and let $y_a,y_b,y_c$ be the vertices adjacent to $x$ on the paths. By minimality of distance, $y_a,y_b,y_c$ are not all equal.

Suppose that $y_a\neq y_b$. Then consider the meridian $\ell$ which bisects $xy_a$ and $xy_b$, and let $\gamma$ be the associated curve. The cycle $K$ consisting of $P_a$, $P_b$ and $ab$ crosses $\gamma$ at $x$, so must cross it at least once more, either because there is some vertex $z\neq x$ on $K$ which is an element of $\ell$, or because there is some edge $zz'$ of $K$ which is an element of $\ell$. If there are multiple crossings, note that each crossing corresponds to an element of $\ell$, and we fix the crossing for which this element $z$ or $zz'$ is closest to the element $x$. Since the length of the cycle is at least $5$, neither $z$ nor $z'$ is adjacent to $x$. 

Without loss of generality $z$ lies on $P_a$, and so an initial section of $P_a$ is a shortest path to $z$. Also, using Lemma \ref{following}, any $x$-$z$ path which follows $\ell$ is a shortest path. Let $e_i$ be the next element of $\ell$ after $x$ used by such a path. Since $x$ and $z$ are not adjacent, for each $y\in e_i$ there is an $x$-$z$ path following $\ell$ which uses $y$. By Lemma \ref{following}, $y_a\in e_i$. Since $y_a\neq y_b$ and $e_i$ is a midpoint of $y_a$ and $y_b$, we must have $e_i=\{y_a,y_b\}$. Consequently there is an alternative shortest path $P'_a$ with $y'_a=y_b$.

Applying the above argument to each pair of paths, each pair of $y_a,y_b,y_c$ are equal or adjacent. Suppose without loss of generality that $y_a=y_b\sim y_c$. Now apply the above argument to the pairs $(P_a,P_c)$ and $(P_b,P_c)$ respectively. If either of the vertices $z$ obtained lies on $P_c$, there is an alternative shortest path $P'_c$ passing through $y_a$; if not, there are alternative shortest paths $P'_a$ and $P'_b$ both passing through $y_c$. In either case this contradicts minimality of the supposed counterexample.
\end{proof}

\subsection{Nice configurations}\label{sec:nice}

We next define the main configurations we use in the proof; we shall show that such a configuration can be identified from the pairwise distances of its boundary vertices, and allows a reduction to a smaller near-triangulation.

For $k\geq 4$, a \textit{nice configuration} of length $k$ in a near-triangulation $T$ consists of $k+1$ consecutive boundary vertices, which we label $v_0,\ldots,v_k$, and $k-2$ internal vertices $w_1\cdots w_{k-2}$, such that the triangles $v_0w_1v_1$, $v_{k-1}w_{k-2}v_{k}$, $v_iw_iv_{i+1}$ for each $1\leq i\leq k-2$, and $w_iw_{i+1}v_{i+1}$ for each $1\leq i\leq k-3$ are all faces of $T$. See Figure \ref{fig:nice} for an example.
\begin{figure}
\centering
\begin{tikzpicture}
\filldraw (0,0) circle (0.075) node [anchor=north] {$v_0$};
\filldraw (1,0) circle (0.075) node [anchor=north] {$v_1$};
\filldraw (2,0) circle (0.075) node [anchor=north] {$\cdots$};
\filldraw (3,0) circle (0.075);
\filldraw (4,0) circle (0.075);
\filldraw (5,0) circle (0.075);
\filldraw (6,0) circle (0.075);
\filldraw (7,0) circle (0.075) node [anchor=north] {$v_k$};
\draw[ultra thick] (0,0) -- (7,0);
\draw[ultra thick, dashed] (-1,0) -- (0,0);
\draw[ultra thick, dashed] (7,0) -- (8,0);
\draw (0,0) -- (1,1.5) -- (6,1.5) -- (7,0);
\draw (1,0) -- (1,1.5) -- (2,0) -- (2.5,1.5) -- (3,0) -- (3.5,1.5) -- (4,0) -- (4.5,1.5) -- (5,0) -- (6,1.5) -- (6,0);
\filldraw [fill=white] (1,1.5) circle (0.075) node [anchor=south] {$w_1$};
\filldraw [fill=white] (2.5,1.5) circle (0.075) node [anchor=south] {$\cdots$};
\filldraw [fill=white] (3.5,1.5) circle (0.075);
\filldraw [fill=white] (4.5,1.5) circle (0.075);
\filldraw [fill=white] (6,1.5) circle (0.075) node [anchor=south] {$w_{k-2}$};
\end{tikzpicture}
\caption{A nice configuration with $k\geq 4$. Boundary edges and vertices are precisely those represented by bold lines and filled circles.}\label{fig:nice}
\end{figure}

We next give a bound on the maximum distance between boundary vertices, which is slightly stronger than the trivial bound (using the path along the boundary) of $\floor{n/2}$.
\begin{lemma}\label{shortcut}Let $T$ be a chordless disc triangulation of boundary length $n>3$ in which every internal vertex has degree at least $6$. Then the maximum distance in $T$ between external vertices is at most $\floor{n/2}-1$.\end{lemma}
\begin{proof}Removing the external vertices leaves a triangulated graph $T'$. Since $T$ contains no chords, $T'$ must be connected and every external vertex of $T$ is adjacent to a vertex on the boundary of $T'$. By \cite[Lemma 2.2]{ABH}, the boundary length of $T'$ is at most $n-6$; note that this counts edges with multiplicity, \ie any edge of $T'$ meeting the external face on both sides is counted twice. The boundary of $T'$ forms a connected graph $H$, and every block of $H$ is either a cycle or a bridge. Writing $a$ for the number of edges of $H$ in cycles and $b$ for the number of bridges of $H$, the bound mentioned is therefore that $a+2b\leq n-6$. Any two vertices $x,y$ on the boundary of $T$ may be connected by taking a shortest path in $H$ between vertices $x'$ adjacent to $x$ and $y'$ adjacent to $y$. Any shortest path in $H$ uses at most half the edges of any cycle, and so has length at most $a/2+b\leq n/2-3$, implying that the path from $x$ to $y$ has length at most $n/2-1$.
\end{proof}

\begin{lemma}\label{nice-distances}Suppose $T$ is a chordless disc triangulation with all internal vertices having degree at least $6$, and $v_0,\ldots,v_k$ are consecutive boundary vertices. Then the following are equivalent.
\begin{itemize}\item$v_0,\ldots,v_k$ are the boundary vertices of a nice configuration.
\item $d_T(v_0,v_k)=k-1$ and $d_T(v_i,v_j)=j-i$ for all other pairs $0\leq i<j\leq k$.
\end{itemize}
\end{lemma}
\begin{proof}First we show the forward implication. Consider the near-triangulation $T'$ obtained by removing $v_1,\ldots,v_{k-1}$. Applying Lemma \ref{geodesic} to $T'$ and the path $v_0w_1\cdots w_{k-2}v_k$, this is a shortest path in $T'$. It follows that for any $0\leq i<j\leq k$, there is a shortest path in $T$ from $v_i$ to $v_j$ which lies entirely within the nice configuration, since any section which goes outside must start and end at vertices on the path $v_0w_1\cdots w_{k-2}v_k$ and lie in $T'$, so can be replaced by a section of that path without increasing the length. The claim follows since these are the distances on the subgraph formed by the nice configuration.

Next we show the reverse implication. The distance conditions ensure that there is a path of length $k-1$ from $v_0$ to $v_k$ which is disjoint from $v_1,\ldots,v_{k-1}$. Write $w_1,\ldots w_{k-2}$ for the vertices on the path other than $v_0$ and $v_k$. We will show that $w_iv_i$ and $w_iv_{i+1}$ are edges of $T$ for each $i$; since $T$ is chordless this means each $w_i$ must be an internal vertex of $T$ and so this is a nice configuration.

Indeed, the cycle $C:=v_0\cdots v_kw_{k-2}\cdots w_1$ can have no chord other than these, since if $j<i$ and $w_iv_j$ is an edge then the path $v_jw_i\cdots w_{k-2}v_k$ is too short unless $(j,i)=(0,1)$, whereas if $j>i+1$ and $w_iv_j$ is an edge then the path $v_0w_1\cdots w_iv_j$ is too short unless $(j,i)=(k-2,k)$. Suppose not all of the required chords of $C$ are edges of $T$. Then there is some chordless cycle $C'$ of length at least $4$ in $T$ whose vertices are vertices of $C$, and each edge of $C'$ is either an edge of $C$ or one of the required chords. We consider the subtriangulation $T'$ with boundary $C'$. If $C'=C$ then Lemma \ref{shortcut} gives that $d_{T'}(v_1,v_k)\leq k-2$ a contradiction. If it contains $v_0,\ldots, v_j$ but not $v_{j+1}$ for $j<k$, then it contains at most $j$ other vertices and Lemma \ref{shortcut} gives $d_{T'}(v_0,v_j)\leq j-1$, a contradiction, and the case $v_k\in C'$ is similar. Finally, if $C'$ contains $v_i,\ldots, v_j$ but not $v_{i-1}$ or $v_{j+1}$ for some $0<i<j<k$ then it contains $w_i,\ldots w_{j-1}$ and possibly also $w_{i-1}$ and/or $w_j$. Lemma \ref{shortcut} gives a contradiction to the distance between $v_i$ and $v_j$ if $w_{i-1},w_j\not\in C'$, whereas if $w_{i-1}\in C'$ it gives a contradiction to the distance between $v_{i-1}$ and $v_j$ via $w_i$ (and the remaining case is similar).
\end{proof}
\begin{lemma}\label{nice-reduce}Let $T$ be a disc triangulation with boundary $v_0,\ldots,v_{n-1}$ and all internal vertices having degree at least $6$, and suppose that $v_0,\ldots,v_k$ and $w_1,\ldots,w_{k-2}$ form a nice configuration. Let $T'$ be the near-triangulation obtained by removing $v_1,\ldots,v_{k-1}$. Then if all pairwise distances between boundary vertices of $T$ are known, we may deduce all pairwise distances between boundary vertices of $T'$.
\end{lemma}
\begin{proof}By Lemma \ref{nice-distances}, $v_0w_1\cdots w_{k-2}v_k$ is a shortest path, and it follows that $d_{T'}(x,y)=d_T(x,y)$ for all $x,y\in T'$. Thus it is sufficient to show that for any $i>k$ we may deduce $(d_T(v_i,w_j))_{1\leq j\leq k-2}$ from $(d_T(v_i,v_j))_{1\leq j\leq k-1}$. We use the following three properties satisfied by the function $f(y)=d_T(x,y)$ for any fixed $x$:
\begin{enumerate}[(i)]
\item if $y$ and $z$ are adjacent then $\abs{f(y)-f(z)}\leq 1$;
\item if $f(y)>0$ then $y$ has a neighbour $z$ with $f(z)=f(y)-1$;
\item $f$ is not constant on any triangle, by Lemma \ref{no-equal-triangle}.
\end{enumerate}
Fix $x=v_i$ for some $i>k$, and let $1\leq j\leq k-2$. 

If $f(v_j)\neq f(v_{j+1})$, then they differ by $1$ and $f(w_j)$ takes one of those two values, using (i). We claim it must be the lower of the two. Suppose not; by symmetry we may assume $f(w_j)=f(v_j)=q$ and $f(v_{j+1})=q-1$. Now applying (ii) to $v_{j+1}$, either $f(w_{j+1})=q-2$ or $f(v_{j+2})=q-2$, but the former is impossible by (i). So $f(w_{j+1})=f(v_{j+1})=q-1$ and $f(v_{j+2})=q-2$. Proceeding in this manner, we see that $f(w_{k-2})=q+2+j-k$ and $f(v_{k-1})=q+1+j-k$. Now applying (ii) to $v_{k-1}$, we must have $f(v_k)=q+j-k$ but this contradicts (i).

If $f(v_j)=f(v_{j+1})=q$, then $f(w_j)=q\pm1$ by (i) and (iii). If $f(v_{j-1})\leq q-1$, then either $j>1$ and the previous paragraph gives $f(w_{j-1})=q-1$, or $j=1$ and $v_0$ is adjacent to $w_1$. If $f(v_{j-1})\geq q$ then by (ii) either we have $f(w_j)=q-1$ or we have $j>1$ and $f(w_{j-1})=q-1$. In all cases either $w_j$ or one of its neighbours has value $q-1$, so $f(w_j)\neq q+1$ by (i).

Thus in every case we may determine $f(w_j)$ from the values $f(v_j)$ and $f(v_{j-1})$.
\end{proof}

\subsection{Additional reductions}\label{sec:misc}
Next we give some additional small cases which yield reductions. Throughout this section we assume $T$ is a chordless disc triangulation with all internal vertices having degree at least $6$; recall that this ensures that no cycle of length less than $6$ encloses any vertices.
\begin{lemma}\label{five}Let $S$ be a set of at least five boundary vertices of $T$. Then the vertices in $S$ have a common neighbour if and only if all pairwise distances between vertices in $S$ are at most $2$.\end{lemma}
\begin{proof}We show the ``if'' statement; ``only if'' is immediate. Write $S=\{w_1,\ldots,w_r\}$, labelled anticlockwise around the boundary cycle. First suppose $r\geq 6$. There must be some vertex $x$ not on the boundary which is adjacent to $w_1$ and $w_4$. There must also be paths of length $2$ between $w_2$ and each of $w_5,\ldots,w_{r-1}$, and between $w_3$ and $w_r$; since there are no chords these must all go via $x$.

If $r=5$ then without loss of generality vertices $w_1$ and $w_5$ are not adjacent (since we have $n\geq 6$). Again there must be a vertex $x$ adjacent to $w_1$ and $w_4$, and this must also be adjacent to $w_2$ and $w_5$. If $w_2,w_3,w_4$ are consecutive boundary vertices then $w_2w_3w_4x$ is a $4$-cycle and must be triangulated by the edge $xw_3$. If not, assume without loss of generality $w_3,w_4$ are not adjacent. Then the path of length $2$ from $w_3$ to $w_5$ must use $x$. In either case $x$ is adjacent to all five vertices.
\end{proof}
\begin{lemma}\label{four}Let $v_1,\ldots,v_4$ be four consecutive boundary vertices of $T$. Then $v_1,\ldots,v_4$ have a common neighbour if and only if $d(v_1,v_4)=2$.\end{lemma}
\begin{proof}Again, it is sufficient to prove the ``if'' statement. Indeed, the path of length $2$ from $v_1$ to $v_4$ uses some vertex $x\neq v_2,v_3$, and now the $5$-cycle $v_1v_2v_3v_4x$ must be triangulated without additional vertices or chords of the boundary cycle, so by the edges $v_2x$ and $v_3x$.
\end{proof}
\begin{lemma}\label{reduce1}Suppose there are five boundary vertices, $w_1,\ldots,w_5$ labelled anticlockwise, with a common internal neighbour $x$. Let $T_i$ be the subtriangulation with boundary consisting of $x$ and the section of the boundary from $w_i$ to $w_{i+1}$ (taking subscripts modulo $5$). Assume the boundary distances of $T$ are given. Then for each $i$, and each pair $y,z$ of boundary vertices of $T_i$ we may deduce $d_{T_i}(y,z)$.
\end{lemma}
\begin{proof}Note that a path between two such vertices in $T$ cannot be shorter than the shortest path in $T_i$: if a section of the path is outside $T_i$ then that section must consist of at least two edges and go between two of $v_i$, $v_{i+1}$ and $x$, but then it may be replaced by edges $v_ix$ and/or $v_{i+1}x$ without increasing the length. Thus it is sufficient to find $d_T(y,z)$ for each such pair. This is known except for pairs of the form $x,y$, where $y$ is on the boundary between $x_i$ and $x_{i+1}$. However, for any such $y$, any shortest path in $T$ between $v_{i+3}$ and $y$ must pass through either $v_i$, $v_{i+1}$ or $x$, and so, since $d(v_{i+3},v_i),d(v_{i+3},v_{i+1})\geq 2$, there is a shortest path which uses the edge $v_{i+3}x$. Thus $d(x,y)=d(v_{i+3},y)-1$.
\end{proof}

\subsection{Boundary distances determine the near-triangulation}\label{sec:main}
We have now established almost all the facts needed to prove the main result. The final piece of the jigsaw is to show that one of the reductions detailed in Sections \ref{sec:nice} and \ref{sec:misc} must be available to us. For this we need the following lemma, which is a modification of \cite[Lemma 15]{HP19}. 

We state the lemma with the weaker condition of a lower bound on the average degree of internal vertices, although when applying it in this section we have a bound on their minimal degree; we shall use this extra strength in Section \ref{sec:counter}. Let $T$ be a disc triangulation with  boundary length $n$. Recall that the boundary faces of $T$ are those which share an edge with the boundary cycle; we endow them with the natural cyclic ordering. Define a \textit{Cleveland vertex} to be an internal vertex lying on exactly two non-consecutive boundary faces (that is, it lies on exactly two boundary faces and those boundary faces are not consecutive). We write $\mathbb{I}_C$ for the indicator function of a condition $C$.
\begin{lemma}\label{layer}Let $T$ be a chordless disc triangulation with boundary length $n$. Suppose the average degree of internal vertices of $T$ is at least $d$. Let $m$ be the number of internal vertices on at least one boundary face, let $k$ be the number of edges incident with the boundary which do not form part of a boundary face, and let $c$ be the number of Cleveland vertices. Then $n\geq (d-5)m+k+c+5+\mathbb{I}_{m\geq 2}$.\end{lemma}
\begin{proof}Note that if $m=1$ then all boundary faces have a common internal vertex, so we must have $k=0$, $c=0$ and $n\geq d$ as all the neighbours of the internal vertex are external. Thus we may assume $m\geq 2$. The fact that the boundary has no chords implies that the subgraph induced by internal vertices adjacent to the boundary is connected.

Adding a new vertex $v^*$ in the unbounded face of $T$ adjacent to all vertices on the boundary gives a triangulation $\tilde T$ of the sphere. with $\abs X+\abs Y+\abs Z+1$ vertices, \ie all faces, including the unbounded one, have degree $3$. It follows (using Euler's formula) that $E=3V-6$, where $F$, $E$ and $V$ denote the number of faces, edges and vertices of $\tilde T$.

We next count the faces of $T$ with at least one boundary vertex. We count these in cyclic order around the boundary: starting from any boundary face including boundary vertices $x,y$, we proceed through all internal faces meeting $y$ (if any) in order, then to the next boundary face meeting boundary vertices $y,z$, and so on. Looking instead at the internal vertices included in faces, the ordering will consist of some boundary faces meeting an internal vertex $u$, followed by a face meeting two internal vertices $u,v$, then possibly some boundary faces meeting $v$, followed by a face meeting $v$ and some other internal vertex $w$, and so on. This gives a cyclic list of internal vertices; however, the same internal vertex may appear in this list more than once. For each entry in the list, we note the number of boundary faces, if any, encountered at the corresponding point in the cyclic ordering. Note that the total number of faces meeting two internal vertices and one boundary vertex is equal to the total length of this list (since each two adjacent entries in the list correspond to one such face). Additionally, every internal vertex which is on a boundary face appears as an entry in this list with a positive number of boundary faces, every Cleveland vertex appears twice as an entry with a positive number of boundary faces, and every edge incident with a boundary vertex which does not form part of a boundary face corresponds to an entry on the list with no boundary faces. Thus the length of the list, and hence the number of faces with exactly one boundary vertex, is at least $m+k+c$.

Additionally, the number of boundary faces is $n$, and so there are at least $n+m+k+c$ faces containing at least one boundary vertex. There is a one-to-one correspondence between such faces and edges meeting the boundary, not counting boundary edges or the new edges of $\tilde T$.

Thus, we may calculate the number of edges of $\tilde T$ as follows. There are $n$ edges between boundary vertices, and another $n$ meeting $v^*$. Summing the degrees of internal vertices counts other edges meeting the boundary once each and other edges not meeting the boundary twice. If there are $m+s$ internal vertices in total, then, we have 
\[d(m+s)+(m+k+c+n)+4n\leq 2E=6(m+s+n+1)-12,\]
and therefore 
\[n\geq (d-5)m+k+c+s+6\geq (d-5)m+k+c+6.\qedhere\] 
\end{proof}

We are now ready to complete the proof.
\begin{proof}[Proof of Theorem \ref{main:tri}]
We proceed by induction on the number of boundary vertices. Recall that if $T$ has a chord we may reduce to two smaller cases. Likewise, if any internal vertex has five or more boundary neighbours then we may identify these neighbours by Lemma \ref{five}, and reduce (using Lemma \ref{reduce1}) to one or more subtriangulations, all of which have strictly fewer than $n$ boundary vertices, thus completing the inductive step of the proof.

If any internal vertex has four consecutive boundary neighbours, we may identify them by Lemma \ref{four}. Suppose that $w$ has consecutive boundary neighbours $v_1,v_2,v_3,v_4$. We claim that in this case we can deduce the distance between boundary vertices in $T\setminus\{v_2,v_3\}$; it suffices to check we can deduce $d_T(v_i,w)$ for each $i>4$. As in the proof of Lemma \ref{nice-reduce}, if the distances from $v_i$ to $v_2$ and $v_3$ differ, say they are $q,q+1$ in some order, then one of $v_1,v_4$ must be at distance $q-1$, so the distance to $w$ must be $q$. If both distances are $q$ then at least one of $v_1$ or $w$ is at distance $q-1$, and since the distance to $w$ is not $q$ by Lemma \ref{no-equal-triangle}, it must be $q-1$.

Thus we may reduce to a smaller case if any of the above configurations exist. If $T$ contains no chord and no internal vertex adjacent to four consecutive boundary vertices or to any five boundary vertices, we will show that it must contain a nice configuration. Then, by Lemma \ref{nice-distances}, we may identify such a configuration, and, by Lemma \ref{nice-reduce}, reduce to a near-triangulation with a shorter boundary.

To establish this, we apply Lemma \ref{layer} with $d=6$. By assumption, there is no internal vertex on more than two boundary faces. Let $m_1$ be the number of internal vertices on one boundary face, let $m_2$ be the number of internal vertices on two consecutive boundary faces, and let $c$ be the number of Cleveland vertices. The result of that lemma gives $n=m_1+2m_2+2c\geq (m_1+m_2+c)+c+k+5$, where $k$ is the number of edges meeting the boundary but not on a boundary face (note that, by assumption, such an edge meets the boundary only once). Thus $m_2\geq k+5$, so as we traverse the boundary we must at some point find two pairs of consecutive boundary faces, with each pair having an internal vertex in common, such that every edge meeting the boundary between them is on a boundary face. This is a nice configuration, completing the proof.
\end{proof}

\section{Disc quadrangulations}
In this section we prove Theorem \ref{main:quad}. The proof follows the same lines as that of Theorem \ref{main:tri} but is slightly simpler, largely due to identification of boundary distances in the reduced near-quadrangulation being more straightforward. Again, by Lemma \ref{lem:hamilton}, we may assume that the cyclic ordering of boundary vertices is known.

Let $Q$ be a chordless disc quadrangulation of boundary length $n$ (where, necessarily, $n$ is even) with all internal vertices having degree at least $4$, and having at least one internal vertex. We first prove a tight bound on the number of edges meeting the boundary.
\begin{lemma}\label{e-circ}Let $e^\circ$ be the number of internal edges meeting one boundary vertex (necessarily exactly one, since $Q$ is chordless). Then $e^\circ\leq n-4$.\end{lemma}
\begin{proof}Write $f,e,v$ for the number of faces (including the external face), edges and vertices of $Q$. Note that summing degrees of internal vertices of $Q$ counts each internal edge meeting a boundary vertex once and counts each other internal edge twice. Thus $2(e-n)-e^\circ\leq 4(v-n)$, \ie $2e\leq 4v-2n+e^\circ$. Summing face degrees gives $4(f-1)+n=2e$, \ie $f=e/2-n/4+1$. Thus Euler's formula gives $e=v+e/2-n/4-1$. Combining these facts, we obtain $e^\circ\leq n-4$. 
\end{proof}
Let $Q'$ be the subgraph obtained by removing all boundary vertices of $Q$. Note that chordlessness implies that $Q'$ is connected. We now give the corresponding result for near-quadrangulations of \cite[Lemma 2.2]{ABH}. Note that, if desired, the equivalent result where chords are permitted may be easily deduced by breaking $Q$ along chords into smaller chordless near-quadrangulations. 
\begin{lemma}\label{quad-layer}Let $m$ be the degree of the external face of $Q'$, \ie the boundary length of $Q'$, counting edges with multiplicity, and let $r$ be the number of boundary vertices of $Q'$. We have $m\leq n-8$, and consequently $r\leq n-8+\mathbb{I}_{n=8}$.\end{lemma}
\begin{proof}Let $F^\circ$ be the set of internal faces of $Q$ meeting at least one boundary vertex, let $E^\circ$ be the set of edges of $Q$ meeting a boundary vertex, and write $f^\circ,e^\circ$ for their sizes. Consider a face $x\in F^\circ$. If $x$ contains two consecutive boundary edges then both remaining edges of $x$ are necessarily in the set $E^\circ$ of Lemma \ref{e-circ} (since if one were a boundary edge, the other would be a chord); otherwise if $x$ contains a boundary edge then both adjacent edges of $x$ are in $E^\circ$. If $x$ does not contain a boundary edge then both edges of $x$ meeting a boundary vertex are in $E^\circ$. Thus every face in $F^\circ$ contains at least two edges of $E^\circ$ and so $f^\circ\leq e^\circ$.

Now the $4f^\circ$ pairs consisting of a face in $F^\circ$ and an edge of that face count the boundary edges of $Q'$ (with multiplicity), the boundary edges of $Q$ (once each) and the edges of $E^\circ$ (twice each). Thus $m=4f^\circ-2e^\circ-n\leq 2f^\circ-n\leq n-8$, as required. Finally, since the boundary of $Q'$ is a connected graph with $r$ vertices, it has at least $r-1$ edges, and if it has exactly $r-1$ edges all are bridges and hence counted twice. It follows that $r\leq m\leq n-8$ unless $m=0$ and $r=1$ (when $n\geq 8$).
\end{proof}
In particular, note that there can be no internal vertices if $n<8$.
\begin{lemma}\label{shortcut-quad}For $n\geq 6$, let $Q$ be a chordless disc quadrangulation with boundary length $n$ such that every internal vertex has degree at least $4$, and let $x$ and $y$ be boundary vertices. Then $d_Q(x,y)\leq\floor{n/2}-\mathbb{I}_{\deg(x)>2}-\mathbb{I}_{\deg(y)>2}$.\end{lemma}
\begin{proof}It is sufficient to prove the case $\deg(x),\deg(y)>2$. The other cases follow from this, since $d_Q(x,y)\leq d_Q(x',y)+1$, where $x'$ is a boundary vertex adjacent to $x$, and chordlessness implies that $\deg(x')>2$ whenever $\deg(x)=2$.

If $\deg(x),\deg(y)>2$ then let $x''$ and $y''$ be internal neighbours of $x$ and $y$ respectively. As in the proof of Lemma \ref{shortcut}, $d_{Q'}(x'',y'')\leq\floor{m/2}$, where $m$ is the length (with multiplicity) of the boundary of $Q'$. By Lemma \ref{quad-layer}, $m\leq n-8$, giving the required bound.\end{proof}

For near-quadrangulations, we redefine a nice configuration as follows. For $k\geq4$, a \textit{nice configuration} of length $k$ in a near-quadrangulation $Q$ consists of $k+1$ consecutive boundary vertices, which we label $v_0,\ldots,v_k$, and $k-3$ internal vertices $w_2\cdots w_{k-2}$, such that the quadrangles $v_0v_1v_2w_2$, $v_{k-2}v_{k-1}v_kw_{k-2}$ and $w_iv_iv_{i+1}w_{i+1}$ for each $2\leq i\leq k-3$ are all faces of $Q$. See Figure \ref{fig:nice-quad} for an example.
\begin{figure}
\centering
\begin{tikzpicture}
\filldraw (0,1) circle (0.075) node [anchor=north east] {$v_0$};
\filldraw (0,0) circle (0.075) node [anchor=north] {$v_1$};
\filldraw (1,0) circle (0.075) node [anchor=north] {$\cdots$};
\filldraw (2,0) circle (0.075);
\filldraw (3,0) circle (0.075);
\filldraw (4,0) circle (0.075);
\filldraw (5,0) circle (0.075) node [anchor=north] {$v_{k-1}$};
\filldraw (5,1) circle (0.075) node [anchor=north west] {$v_k$};
\draw[ultra thick] (0,1) -- (0,0) -- (5,0) -- (5,1);
\draw[ultra thick, dashed] (-.7,1.7) -- (0,1);
\draw[ultra thick, dashed] (5,1) -- (5.7,1.7);
\draw (0,1) -- (5,1);
\foreach \x in {1,2,3,4}{\draw (\x,1) -- (\x,0);}
\filldraw [fill=white] (1,1) circle (0.075) node [anchor=south] {$w_2$};
\filldraw [fill=white] (2,1) circle (0.075) node [anchor=south] {$\cdots$};
\filldraw [fill=white] (3,1) circle (0.075);
\filldraw [fill=white] (4,1) circle (0.075) node [anchor=south] {$w_{k-2}$};
\end{tikzpicture}
\caption{A nice configuration for a near-quadrangulation, with $k\geq 4$.}\label{fig:nice-quad}
\end{figure}
\begin{lemma}\label{quad-nice-exists}If $Q$ is a chordless disc quadrangulation with boundary length $n>4$ such that all internal vertices have degree are at least $4$, then $Q$ contains a nice configuration.
\end{lemma}
\begin{proof}By Lemma \ref{e-circ} the number of internal edges meeting an external vertex is at most $n-4$. Writing $S_0,S_1,S_{\geq 2}$ for the sets of external vertices with zero, one and at least two internal neighbours respectively, we have $\abs{S_0}\geq\abs{S_{\geq2}}+4$. Consequently the vertices of $S_0$ divide the boundary into $\abs{S_0}\geq 4$ intervals, some of which contain no vertices of $S_{\geq2}$ (and each of which has at most $n-3$ edges). Consequently, relabelling if necessary, we have, for some $k\geq 3$, distinct consecutive boundary vertices $v_0,\ldots,v_k$ with $v_1,v_{k-1}\in S_0$ and $v_i\in S_1$ for each $2\leq i\leq k-2$. We cannot have $k=3$, since then $v_0v_1v_2v_3$ would be a face, making $v_0v_3$ a chord, so $k\geq 4$. For each $2\leq i\leq k-2$, the vertex $v_i$ has a unique internal neighbour $w_i$, and together with $v_0,\ldots,v_k$ these form a nice configuration.
\end{proof}

We now need to show that a nice configuration may be identified from the information available. The proof will be similar to the triangulation case.
\begin{lemma}\label{quad-geodesic}Suppose that $Q$ is a disc quadrangulation with all internal vertices having degree at least $4$, and $v_0,\ldots,v_k$ are consecutive boundary vertices with $\deg(v_i)\geq 3$ for each $1\leq i\leq k-1$. Then $d_Q(v_0,v_k)=k$, and the unique path of length exactly $k$ is $v_0v_1\cdots v_k$.
\end{lemma}
\begin{proof}
First we argue that if $0\leq i<j\leq k$ then $v_i,v_j$ are not adjacent unless with $j=i+1$, and do not have a common neighbour. Indeed, suppose not and take a counterexample with $j-i$ as small as possible. Now either $v_iv_{i+1}\cdots v_j$ (if $v_i$ and $v_j$ are adjacent) or $v_iv_{i+1}\cdots v_jw$ (if they have a common neighbour $w$) forms a cycle $C$, and $C$ together with the vertices inside it induces a quadrangulation of $C$, which is chordless by minimiality of the counterexample. However, each $\ell$ with $i<\ell<j$ meets at least one internal edge of this quadrangulation of $C$, so at least $\abs{C}-3$ internal edges meet boundary vertices, contradicting Lemma \ref{e-circ}.

Now we prove the desired result. Suppose it is not true and take a counterexample with as few total vertices in the near-quadrangulation as possible, and among such counterexamples, with $k$ as small as possible. Let $P$ be either a shorter path or a different path of length $k$. If $v_j$ is on $P$ for some $1\leq j<k$ then $P$ contains either a $v_0$-$v_j$ path of length at most $j$ other than $v_0v_1\cdots v_j$ or a $v_j$-$v_k$ path of length at most $k-j$ other than $v_jv_{j+1}\cdots v_k$, contradicting minimality in each case. Thus $P$ avoids $v_1,\ldots,v_{k-1}$.

Consider the set of faces incident with at least one of $v_1,\ldots,v_{k-1}$. The number of such faces is $k+\sum_{i=1}^{k-1}(\deg(v_i)-3)\geq k$. There is another path $P'$ from $v_0$ to $v_k$ around the other side of this strip of faces, which has length at least $k+2$ and is disjoint from $\{v_1,\ldots,v_{k-1}\}$; in particular $P'$ is longer than $P$, so is not a shortest path in $Q\setminus\{v_1,\ldots,v_{k-1}\}$. Suppose that $P'$ consists entirely of internal vertices of $Q$. Then removing $v_1,\ldots,v_{k-1}$ leaves a disc quadrangulation $Q'$ for which $P'$ forms part of the boundary, and since each internal vertex of $P'$ was adjacent to at most one of the removed vertices it has degree at least $3$ in the remaining near-quadrangulation. Since $P'$ is not a shortest path between $v_0$ and $v_k$ in $Q'$, this contradicts minimality of the original example. 

If some vertices of $P'$ are boundary vertices of $Q$, we may remove any bridges from $Q'$ and split any cutvertices to leave two or more disc quadrangulations and apply the same argument to each section of $P'$ that remains, noting that each vertex of $P'$ which lies strictly inside such a section has degree at least $3$.
\end{proof}

\begin{lemma}\label{quad-nice-dist}Suppose $Q$ is a chordless disc quadrangulation with all internal vertices having degree at least $4$, and $v_0,\ldots,v_k$ are consecutive boundary vertices, where $k\geq 4$. Then the following are equivalent.
\begin{itemize}\item $v_0,\ldots,v_k$ are the boundary vertices of a nice configuration.
\item $d_Q(v_0,v_k)=k-2$ and $d_Q(v_i,v_j)=j-i$ for all other pairs $0\leq i<j\leq k$.
\end{itemize}
\end{lemma}
\begin{proof}We first prove the forward implication; write $w_2,\ldots w_{k-2}$ for the internal vertices of the nice configuration.

If $0<i<j<k$ this follows directly by applying Lemma \ref{quad-geodesic} to $Q$, since $\deg(v_\ell)=3$ for all $i<\ell<j$. If $i=0$ and $j<k$, we deduce this by induction on $j$ (it is clearly true for $j=1$). Suppose $j>1$ and let $Q'$ be the near-quadrangulation obtained by removing $v_1,\ldots,v_{j-1}$; since each $w_\ell$ is internal for $Q$, this is a disc quadrangulation. Since $v_0,w_2,\ldots,w_j,v_j$ forms part of the boundary of $Q'$, and $\deg_{Q'}(w_{\ell})\geq3$ for $2\leq\ell\leq j$, Lemma \ref{quad-geodesic} gives $d_{Q'}(v_0,v_j)=j$. Since $d_Q(v_0,v_\ell)=\ell$ and $d_Q(v_\ell,v_j)=j-\ell$ (using the induction hypothesis and the previous case) for each $0<\ell<j$, there is no shorter path via $v_\ell$ and so $d_Q(v_0,v_j)=j$. The case $i>0$ and $j=k$ is similar. Finally, $v_0w_2\cdots w_{k-2}v_k$ is a path of length $k-2$, and we cannot have $d_Q(v_0,v_k)<k-2$ since $d_Q(v_0,v_{k-1})=k-1$.

Next we show the reverse implication. The distance conditions ensure that there is a path of length $k-2\geq 2$ from $v_0$ to $v_k$ which is disjoint from $v_1,\ldots,v_{k-1}$. Write $w_2,\ldots w_{k-2}$ for the vertices on the path other than $v_0$ and $v_k$; for ease of notation we also set $w_1=v_0$ and $w_{k-1}=v_k$. Since $Q$ is chordless, each $w_i$ is internal for $2\leq i\leq k-2$. It suffices to show that $v_iw_i$ is an edge of $Q$ for $2\leq i\leq k-2$.

Indeed, the cycle $C:=v_0\cdots v_kw_{k-2}\cdots w_2$ can have no chord other than these, since if $j\neq i$ and $w_iv_j$ is an edge then either the path $v_jw_i\cdots w_{k-2}v_k$ or the path $v_0w_2\cdots w_iv_j$ contradicts a distance condition. Suppose not all of the required chords of $C$ are edges of $Q$. Then there is some chordless cycle $C'$ of length at least $6$ in $T$ whose vertices are vertices of $C$, and the boundary of $C'$ consists of $v_iw_i$, $v_jw_j$ and the paths $w_iw_{i+1}\cdots w_j$ and $v_iv_{i+1}\cdots v_j$ for some $1\leq i< j\leq k-1$.

We consider the subquadrangulation $Q'$ with boundary $C'$. Since $Q'$ is chordless, we have either $\deg_{Q'}(v_i)\geq 3$ or $\deg_{Q'}(w_i)\geq 3$ (or both), and similarly for $v_j,w_j$. Thus Lemma \ref{shortcut-quad} gives that one of the four distances $d_{Q'}(v_i,v_j)$, $d_{Q'}(w_i,v_j)$, $d_{Q'}(v_i,w_j)$ or $d_{Q'}(w_i,w_j)$ is at most $j-i-1$. However, in each case this contradicts a distance condition: respectively, the condition $d_Q(v_i,v_j)=j-i$, $d_Q(v_0,v_j)=j$, $d_Q(v_i,v_k)=k-j$ or $d_Q(v_0,v_k)=k-2$.
\end{proof}

Next we show that identifying a nice configuration will allow us to reduce the problem of determining the near-quadrangulation to a smaller case.
\begin{lemma}\label{quad-nice-reduce}Let $v_0,\ldots, v_k$ and $w_2,\ldots w_{k-2}$ be the vertices of a nice configuration in $Q$, and let $Q'$ be the near-quadrangulation obtained by removing $v_1,\ldots,v_{k-1}$. Let $z$ be any vertex of $Q'$. Then $d_{Q'}(w_i,z)=d_Q(v_i,z)-1$ for each $2\leq i\leq k-2$.
\end{lemma}
\begin{proof}Note first that if $x,y\in V(Q')$ then $d_{Q'}(x,y)=d_Q(x,y)$. This is true when $x,y\in\{v_0,w_2,\ldots,w_{k-2},v_k\}$ by Lemma \ref{quad-nice-dist}, since a shorter route in $Q$ between any two of these vertices would contradict the fact that $d_Q(v_0,v_k)=k-2$. The corresponding fact for arbitrary $x,y$ follows since any path which leaves $Q'$ contains a subpath between two vertices of $\{v_0,w_2,\ldots,w_{k-2},v_k\}$, which can be replaced by a shortest path in $Q'$ between those two vertices.

Now suppose the desired result is not true, and for fixed $z$ choose $i$ for which $d_{Q'}(w_i,z)\neq d_Q(v_i,z)-1$ with $d_{Q}(v_i,z)$ as small as possible. Since there is a path in $Q$ from $v_i$ to $z$ of length $1+d_{Q'}(w_i,z)$, we must have $d_{Q}(w_i,z)=d_{Q'}(w_i,z)>d_Q(v_i,z)-1$ and so the shortest path in $Q$ from $v_i$ to $z$ does not use $w_i$. Consequently we have either $d_Q(v_{i-1},z)=d_Q(v_i,z)-1$ or $d_Q(v_{i+1},z)=d_Q(v_i,z)-1$; assume without loss of generality the former. If $i-1>1$ then $d_{Q'}(w_{i-1},z)\geq d_{Q'}(w_i,z)-1>d_Q(v_{i-1},z)-1$, contradicting our choice of $i$. Otherwise (if $i=2$), we must have $d_{Q'}(v_0,z)=d_Q(v_0,z)=d_Q(v_1,z)-1\leq d_Q'(w_2,z)-2$, but this gives a contradiction since $d_{Q'}(v_0,z)\geq d_{Q'}(w_2,z)-1$.
\end{proof}

We now have all the ingredients needed to complete the proof.
\begin{proof}[Proof of Theorem \ref{main:quad}]
We proceed by induction on the boundary length, $n$. By Lemma \ref{quad-layer}, if $n<8$ then there are no internal vertices and we may identify $Q$ from the pairs at distance $1$. If $Q$ has a chord, we may divide $Q$ along the chord into two subquadrangulations $Q_1$ and $Q_2$, for which we know all boundary distances, and hence by the induction hypothesis can determine in full. 

Thus we may assume $Q$ is chordless and $n\geq 8$. By Lemma \ref{quad-nice-exists} there is a nice configuration, with boundary vertices $v_0,\ldots, v_k$, say, which we may identify by Lemma \ref{quad-nice-dist}. By Lemma \ref{quad-nice-reduce}, together with the induction hypothesis, we may deduce all boundary distances for the near-quadrangulation obtained by removing $v_1,\ldots, v_{k-1}$, and hence determine it precisely. Together with the known edges meeting $v_1,\ldots,v_{k-1}$, this determines $Q$.
\end{proof}

\section{Mixing triangles and quadrangles}\label{sec:counter}

Noting that the conditions of minimum degree $6$ for triangles and $4$ for quadrangles correspond to locally everywhere non-positive curvature, we might conjecture that a similar condition is sufficient to determine planar graphs where all internal faces are triangles or quadrangles by the distances between boundary vertices. In particular, the natural condition to impose on internal vertices is that $2t(v)+3q(v)\geq 12$, where $t(v),q(v)$ are the number of triangular and quadrangular faces containing $v$ respectively. We first show that this has similar properties to the conditions on disc triangulations and quadrangulations, in that it excludes examples of the kind discussed in Section \ref{sec:intro}, that is, obtained by adding extra vertices inside a face.

\begin{prop}Let $G$ be a plane graph with a simple closed boundary of length $n$ for which every internal face is a triangle or quadrangle. Suppose that $G$ has at least one internal vertex, and every internal vertex satisfies $2t(v)+3q(v)\geq 12$. Then $n\geq 6$.
\end{prop}
\begin{proof}We may assume that $G$ is chordless, since any counterexample with a chord must contain a chordless counterexample as a subgraph. Note that each internal vertex $v$ has $\deg(v)=t(v)+q(v)\geq 6-q(v)/2$. Consequently, denoting the set of internal vertices by $I(G)$, we have $\sum_{v\in I(G)}\deg(v)\geq 6\abs{I(G)}-q_1/2-q_2-3q_3/2-2q_4$, where $q_i$ is the number of quadrangles containing exactly $i$ internal vertices. We now add a diagonal to each quadrangle to form a near-triangulation $T$. If a quadrangle contains $i$ internal vertices, we may ensure that the diagonal we add meets at least $\lceil i/2\rceil$ internal vertices, and this also ensures that we do not create any new chords (the fact that $G$ is chordless means that every quadrangle meets at least one new vertex). Thus we have 
\[\sum_{v\in I(T)}\deg_T(v)\geq\sum_{v\in I(G)}\deg_G(v)+q_1+q_2+2q_3+2q_4\geq 6\abs{I(T)},\]
and so the average degree of internal vertices of $T$ is at least $6$. Now Lemma \ref{layer} gives the desired result.
\end{proof}

\begin{figure}
\centering
\begin{tikzpicture}[thick]
\draw (0:0) -- (0:1) -- ++(30:1) coordinate (a) -- ++(-90:1) coordinate (b) -- ++(150:1) -- ++(-120:1);
\draw (0:0) -- (-60:1) -- ++(-30:1) coordinate (c) -- ++(-150:1) coordinate (d) -- ++(90:1) -- ++(-180:1);
\draw (0:0) -- (-120:1) -- ++(-90:1) coordinate (e) -- ++(-210:1) coordinate (f) -- ++(30:1) -- ++(120:1);
\draw (0:0) -- (180:1) -- ++(210:1) coordinate (g) -- ++(90:1) coordinate (h) -- ++(-30:1) -- ++(60:1);
\draw (0:0) -- (120:1) -- ++(150:1) coordinate (i) -- ++(30:1) coordinate (j) -- ++(-90:1) -- ++(0:1);
\draw (0:0) -- (60:1) -- ++(90:1) coordinate (k) -- ++(-30:1) coordinate (l) -- ++(210:1) -- ++(-60:1);
\foreach \x in {a,b,c,d,e,f,g,h,i,j,k,l}{\filldraw (\x) circle (0.075);}
\draw (b) -- (c);
\draw (d) -- (e);
\draw (f) -- (g);
\draw (h) -- (i);
\draw (j) -- (k);
\draw (l) -- (a);
\foreach \x in {0,60,120,180,240,300}{\filldraw (\x:1) circle (0.075);}
\filldraw (0:0) circle (0.075);
\node[anchor=north] at (e) {$u$};
\node[anchor=north] at (d) {$v$};

\draw (0:5) -- ++(0:1) coordinate (u1) -- ++(30:1) coordinate (a1) -- ++(-90:1) coordinate (b1) -- ++(150:1) -- ++(-120:1);
\draw (0:5) -- ++(-60:1) coordinate (v1) -- ++(-30:1) coordinate (c1) -- ++(-150:1) coordinate (d1) -- ++(90:1) -- ++(-180:1);
\draw (0:5) -- ++(-120:1) coordinate (w1) -- ++(-90:1) coordinate (e1) -- ++(-210:1) coordinate (f1) -- ++(30:1) -- ++(120:1);
\draw (0:5) -- ++(180:1) coordinate (x1) -- ++(210:1) coordinate (g1) -- ++(90:1) coordinate (h1) -- ++(-30:1) -- ++(60:1);
\draw (0:5) -- ++(120:1) coordinate (y1) -- ++(150:1) coordinate (i1) -- ++(30:1) coordinate (j1) -- ++(-90:1) -- ++(0:1);
\draw (0:5) -- ++(60:1) coordinate (z1) -- ++(90:1) coordinate (k1) -- ++(-30:1) coordinate (l1) -- ++(210:1) -- ++(-60:1);
\foreach \x in {a1,b1,c1,d1,e1,f1,g1,h1,i1,j1,k1,l1}{\filldraw (\x) circle (0.075);}
\draw (b1) -- (c1);
\draw (d1) -- (e1);
\draw (f1) -- (g1);
\draw (h1) -- (i1);
\draw (j1) -- (k1);
\draw (l1) -- (a1);

\draw (c1) -- ++(-90:1) coordinate (aa1) -- ++(-150:1) coordinate (ab1) -- ++(-180:1) coordinate (ac1) -- ++(150:1) coordinate (ad1) -- (f1);
\draw (aa1) -- (d1) -- (ab1);
\draw (ac1) -- (e1) -- (ad1);
\foreach \x in {u1,v1,w1,x1,y1,z1,aa1,ab1,ac1,ad1}{\filldraw (\x) circle (0.075);}
\filldraw (0:5) circle (0.075);

\draw (0:10) -- ++(30:1) coordinate (u2) -- ++(60:1) coordinate (a2) -- ++(-60:1) coordinate (b2) -- ++(180:1) -- ++(-90:1);
\draw (0:10) -- ++(-30:1) coordinate (v2) -- ++(0:1) coordinate (c2) -- ++(-120:1) coordinate (d2) -- ++(120:1) -- ++(-150:1);
\draw (0:10) -- ++(-90:1) coordinate (w2) -- ++(-60:1) coordinate (e2) -- ++(-180:1) coordinate (f2) -- ++(60:1) -- ++(150:1);
\draw (0:10) -- ++(210:1) coordinate (x2) -- ++(240:1) coordinate (g2) -- ++(120:1) coordinate (h2) -- ++(0:1) -- ++(90:1);
\draw (0:10) -- ++(150:1) coordinate (y2) -- ++(180:1) coordinate (i2) -- ++(60:1) coordinate (j2) -- ++(-60:1) -- ++(30:1);
\draw (0:10) -- ++(90:1) coordinate (z2) -- ++(120:1) coordinate (k2) -- ++(0:1) coordinate (l2) -- ++(240:1) -- ++(-30:1);
\foreach \x in {a2,b2,c2,d2,e2,f2,g2,h2,i2,j2,k2,l2}{\filldraw (\x) circle (0.075);}
\draw (b2) -- (c2);
\draw (d2) -- (e2);
\draw (f2) -- (g2);
\draw (h2) -- (i2);
\draw (j2) -- (k2);
\draw (l2) -- (a2);

\draw (d2) -- ++(-90:1) coordinate (aa2) -- ++(-150:1) coordinate (ab2) -- ++(-180:1) coordinate (ac2) -- ++(150:1) coordinate (ad2) -- (g2);
\draw (aa2) -- (e2) -- (ab2);
\draw (ac2) -- (f2) -- (ad2);
\foreach \x in {u2,v2,w2,x2,y2,z2,aa2,ab2,ac2,ad2}{\filldraw (\x) circle (0.075);}
\filldraw (0:10) circle (0.075);
\end{tikzpicture}
\caption{A graph with zero curvature whose boundary distances are preserved by relabelling (left), and two non-isomorphic graphs with identical boundary distances (middle and right).}\label{fig:counter}
\end{figure}
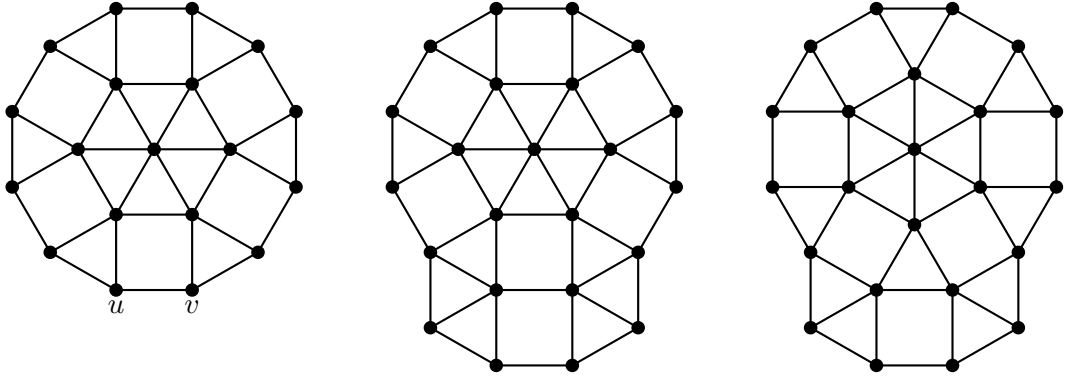

However, this condition is not sufficient to determine the graph from its boundary distances, as the graphs in Figure \ref{fig:counter} show. The left-hand graph has $2t(v)+3q(v)=12$ for every internal vertex $v$, but has the property that each boundary vertex is at distance $1$, $2$ and $3$ respectively from the next three vertices in either direction, and at distance $4$ from the five remaining boundary vertices. In particular, the boundary distances are preserved by a cyclic relabelling of the boundary. Consequently, boundary distances cannot distinguish between the original example and the isomorphic configuration for which $uv$ lies on a triangular face if either appears anywhere in a planar graph. By adding some extra faces we can obtain two non-isomorphic graphs that satisfy the condition on internal vertices but nevertheless have identical boundary distances, for example the middle and right-hand graphs in Figure \ref{fig:counter}. The original example, and one of the non-isomorphic pair but not the other, may be obtained as a section of the rhombitrihexagonal tiling after dividing each hexagon into six triangles.

\section{Algorithmic considerations}\label{sec:algo}
Suppose we are given a set of boundary vertices and all distances between them. The proof of our main result gives a polynomial-time algorithm to determine whether there exists a near-triangulation with that boundary and all internal vertices having degree at least $6$ that achieves those distances (and in fact reconstructs the unique such graph, if one exists). One might naturally ask whether the decision problem can still be solved in polynomial time if we relax the minimum degree condition. In this case a near-triangulation, if it exists, will not be unique, but can we efficiently determine whether it exists? Observe that removing the degree condition gives a genuinely different question, since without it Lemma \ref{shortcut} is not necessarily satisfied.

Similar problems arise in other contexts, particularly for the identification and reconstruction of phylogenetic networks. A (binary) phylogenetic network is a rooted graph where every vertex is either a tree vertex, with two children and (unless it is the root) one parent; a leaf with no children and one parent; or a reticulation vertex with one child and two parents. Such networks represent divergence between taxa in evolutionary history, with reticulation vertices corresponding to hybridisation events. 

Reconstruction of such networks from partial information is a key task in evolutionary biology. In particular, distance-based methods which seek to reconstruct the network from information about the lengths of paths between taxa (that is, leaves) are widely studied; here we may think of the set of leaves as the boundary of the network.
Early work focussed on the simplest case of a phylogenetic tree, that is, a phylogenetic network without reticulation points. Simple necessary and sufficient conditions for a phylogenetic tree realising a set of boundary distances to exist, in terms of a relationship satisfied by the distances between any four vertices, were found by Zarecki\u{\i} \cite{Zar65} and subsequently rediscovered by Buneman \cite{Bun74}. These apply to both weighted and unweighted trees, with an additional parity constraint for the latter, and give a straightforward quartic recognition algorithm. 

If weighted networks are considered, one can always modify some weights while retaining all distances between leaves (for example, by choosing an internal vertex and transferring a fixed small amount of weight from each incoming edge to each outgoing edge), and so reconstruction is only possible up to equivalence under these and other transformations. Consequently we essentially have a recognition problem rather than reconstruction, though of course it is interesting to classify the precise equivalence up to which reconstruction is possible. This has been accomplished recently in various contexts. Bordewich and Semple \cite{BS16} showed that under the assumption that the network is tree-child, that is, every non-leaf vertex has at least one child which is either a leaf or a tree vertex, knowing the lengths of \emph{all} paths between each pair of leaves is sufficient to reconstruct a network efficiently, and this was subsequently extended to weighted networks \cite{BST18}. Results closer to ours, in which only the shortest-path distance between each pair of leaves is given, have been obtained with the additional assumption of either an ultrametric tree-child network \cite{BT16} or a normal tree-child network \cite{BHMS18}. A polynomial-time recognition algorithm was also recently obtained for metrics arising from cactus graphs \cite{HHMM}.

Returning to the question of recognising near-triangulations, it actually makes no difference whether we ask for a near-triangulation or, more simply, for a plane graph which achieves the given boundary distances. This is because any plane graph with given boundary can be extended to a triangulation of that boundary, without changing any distances between existing vertices. To do this, if an internal face has $r>3$ edges, add a band of $2r$ triangles inside that face, leaving another $r$-face inside the band. Repeat this process $\floor{r/4}$ times, and then triangulate the final $r$-face ad lib. Notice that the addition of a band does not change the distance between two original vertices, and the edges added in the final step are too far away from the original vertices to introduce a shorter path. Thus we ask the following.

\begin{question}Given a set $V$ of vertices and all pairwise distances between them, is there a polynomial-time algorithm to determine whether there is a plane graph with simple boundary consisting of $V$ in some cyclic ordering that realises those distances?
\end{question}
Note that, as before, Lemma \ref{lem:hamilton} means that we may assume the cyclic ordering of $V$ is given.

There are some non-trivial necessary conditions on the distances for such a graph to exist (that is, conditions which would not be necessary if the planarity requirement were dropped). In particular, we have the following condition satisfied by any four points (which can easily be checked in polynomial time).
\begin{prop}\label{four-point}If $a,b,c,d$ are boundary vertices appearing in that order, then $d(a,b)+d(c,d)\leq d(a,c)+d(b,d)$.
\end{prop}
\begin{proof}Note that if $P_1$ is an $a$-$c$ path and $P_2$ is a $b$-$d$ path then $P_1$ and $P_2$ must intersect. Choose $P_1$ and $P_2$ to be shortest paths, and let $x$ be a point of intersection. Now since $x$ is on each shortest path we have $d(a,c)=d(a,x)+d(x,c)$ and $d(b,d)=d(b,x)+d(x,d)$. However, $d(a,b)\leq d(a,x)+d(x,b)$ and $d(c,d)\leq d(c,x)+d(x,d)$, giving the result.
\end{proof}
However, Proposition \ref{four-point} (together with the distances forming an integer-valued metric) is not sufficient to ensure planarity. For example, boundary distances for the graph shown in Figure \ref{fig:non-plane} satisfy Proposition \ref{four-point}, but no plane graph has the same boundary and boundary distances, since the only way to achieve $d(v_0,v_3)=d(v_1,v_6)=d(v_4,v_7)=2$ without crossing edges or adding external vertices is to identify $x$ and $y$, reducing $d(v_0,v_4)$.

\begin{figure}[ht]
\centering
\begin{tikzpicture}[thick]
\filldraw (0,0) circle (0.05) node [anchor=north east] {$v_0$};
\filldraw (1,-0.5) circle (0.05) node [anchor=north] {$v_1$};
\filldraw (2,-0.5) circle (0.05) node [anchor=north] {$v_2$};
\filldraw (3,0) circle (0.05) node [anchor=north west] {$v_3$};
\filldraw (3,1) circle (0.05) node [anchor=south west] {$v_4$};
\filldraw (2,1.5) circle (0.05) node [anchor=south] {$v_5$};
\filldraw (1,1.5) circle (0.05) node [anchor=south] {$v_6$};
\filldraw (0,1) circle (0.05) node [anchor=south east] {$v_7$};
\filldraw (1,0) circle (0.05) node [anchor=south east] {$x$};
\filldraw (1.5,1) circle (0.05) node [anchor=north] {$y$};
\draw (0,0) -- (1,-0.5) -- (2,-0.5) -- (3,0) -- (3,1) -- (2,1.5) -- (1,1.5) -- (0,1) -- cycle;
\draw (0,0) -- (3,0);
\draw (1,-0.5) -- (1,1.5);
\draw (0,1) -- (3,1);
\end{tikzpicture}
\caption{A graph whose boundary distances satisfy Proposition \ref{four-point} but are not consistent with a plane graph having the same boundary.}\label{fig:non-plane}
\end{figure}
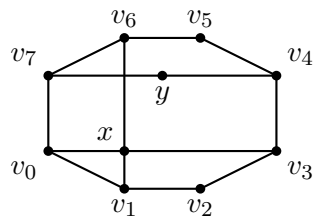

\section*{Acknowledgements}
Research supported by the European Research Council (ERC) under the European Union's Horizon 2020 research and innovation programme (grant agreement no.\ 639046), and by the UK Research and Innovation Future Leaders Fellowship MR/S016325/1. I am grateful to Itai Benjamini, Agelos Georgakopoulos, Christoforos Panagiotis and Alex Scott for many helpful comments at different stages of this work.

\end{document}